\documentclass[11pt]{nyjm} 

\usepackage{amssymb}
\usepackage{hyperref}
\usepackage{cleveref}

\title[Bi-invariant orders on a nilpotent group]{The space of bi-invariant orders on a nilpotent group}

\author{Dave Witte Morris}

\address
{Department of Mathematics and Computer Science,
University of Lethbridge, Lethbridge, Alberta, T1K~3M4, Canada}
\email{Dave.Morris@uleth.ca, 
\href{http://people.uleth.ca/~dave.morris/}{http://people.uleth.ca/\!$\sim$dave.morris/}}



\DeclareMathOperator{\LO}{LO}
\DeclareMathOperator{\VLO}{VLO}
\DeclareMathOperator{\BO}{BiO}
\DeclareMathOperator{\Aut}{Aut}
\DeclareMathOperator{\Out}{Out}
\DeclareMathOperator{\Comm}{Comm}
\DeclareMathOperator{\dom}{domain}

\newcommand{\normal}{\triangleleft}
\newcommand{\lamron}{\triangleright}
\newcommand{\precstar}{\stackrel{\vbox{\hbox{$\scriptstyle*$}\vskip-3pt}}{\prec}}
\newcommand{\succstar}{\stackrel{\vbox{\hbox{$\scriptstyle*$}\vskip-3pt}}{\succ}}

\newcommand{\pref}[1]{(\ref{#1})}
\newcommand{\fullcref}[2]{\cref{#1}\pref{#1-#2}}
\newcommand{\csee}[1]{(see \cref{#1})}

\DeclareMathOperator{\rank}{rank}

\newcommand{\integer}{\mathbb{Z}}
\newcommand{\rational}{\mathbb{Q}}
\newcommand{\real}{\mathbb{R}}
\newcommand{\GG}{\mathbb{G}}
\newcommand{\HH}{\mathbb{H}}
\newcommand{\CC}{\mathbb{C}}

\newcommand{\isol}[1]{\langle \! \langle #1 \rangle \! \rangle}

\newcommand\bigset[2]{\left\{\, #1 
 \mathrel{\left| \vphantom {\left\{ #1 \mid #2 \right\} }
 \right.} #2 \,\right\} }

\numberwithin{equation}{section}
\renewcommand{\thesubsection}{\thesection\Alph{subsection}}
\crefformat{section}{Section~#2#1#3}

\newtheorem{thm}[equation]{Theorem}
\newtheorem{prop}[equation]{Proposition}
\newtheorem{cor}[equation]{Corollary}
\newtheorem{lem}[equation]{Lemma}

\crefmultiformat{prop}{Propositions~#2#1#3}{ and~#2#1#3}{, #2#1#3}{, and~#2#1#3}

\theoremstyle{definition}
\newtheorem{rem}[equation]{Remark}
\newtheorem{rems}[equation]{Remarks}
\newtheorem{defn}[equation]{Definition}
\newtheorem{eg}[equation]{Example}
\newtheorem{notation}[equation]{Notation}
\newtheorem{obs}[equation]{Observation}
\newtheorem*{ack}{Acknowledgments}

\crefformat{rems}{Remark~#2#1#3}

 \newcounter{case}
 \newenvironment{case}[1][\unskip]{\refstepcounter{case}\bf
 \par\medbreak \noindent \hskip\parindent Case \thecase\ #1. \it}{\unskip\upshape}
\crefformat{case}{Case~#2#1#3}

\renewcommand{\Cref}[1]{\cref{#1}}

\makeatletter
\newcommand{\noprelistbreak}{\smallskip\@nobreaktrue\nopagebreak} 
\makeatother

\begin{document}

\vbox to 0pt{\vskip-1in
\centerline{[published in \emph{New York Journal of Mathematics} 18 (2012) 261--273]}
\centerline{\url{http://nyjm.albany.edu/j/2012/18-13.html}}
\vss}

\begin{abstract}
We prove a few basic facts about the space of bi-invariant (or left-invariant) total order relations on a torsion-free, nonabelian, nilpotent group~$G$. For instance, we show that the space of bi-invariant orders has no isolated points (so it is a Cantor set if $G$ is countable), and give examples to show that the outer automorphism group of~$G$ does not always act faithfully on this space. Also, it is not difficult to see that the abstract commensurator group of~$G$ has a natural action on the space of left-invariant orders, and we show that this action is faithful. These results are related to recent work of T.\,Koberda that shows the automorphism group of~$G$ acts faithfully on this space.
\end{abstract}

\maketitle

\setcounter{tocdepth}{1} 
\tableofcontents

\section{Introduction} \label{IntroSect}

\begin{defn} Let $G$ be an abstract group.
	\begin{itemize} \itemsep=\smallskipamount
	\item A total order~$\prec$ on the elements of~$G$ is:
		\begin{itemize}
		\item \emph{left-invariant} if $x \prec y \Rightarrow gx \prec gy$, for all $x,y,g \in G$,
		and
		\item \emph{bi-invariant} if it is both left-invariant and \emph{right-invariant} (which means  $x \prec y \Rightarrow xg \prec yg$, for all $x,y,g \in G$).
		\end{itemize}
	\item The set of all left-invariant orders on~$G$ is denoted $\LO(G)$. 
	It has a natural topology that makes it into a compact, Hausdorff space: for any $x,y \in G$, we have the basic open set $\{\, {\prec} \in \LO(G) \mid x \prec y \,\}$ (see~\cite{Sikora-TopOrders}).
	\item The set of all bi-invariant orders on~$G$ is denoted $\BO(G)$. It is a closed subset of $\LO(G)$.
	\item Any group isomorphism $G_1 \stackrel{\cong}{\to} G_2$ induces a bijection $\LO(G_1) \to \LO(G_2)$. Therefore, the automorphism group $\Aut(G)$ acts on $\LO(G)$. (Furthermore, the subset $\BO(G)$ is invariant.)
	\end{itemize}
\end{defn}

It is known \cite[Thm.~B, p.~1688]{Navas-DynamicsLO} that if $G$ is a locally nilpotent group, and $G$ is not an abelian group of rank\/~$\le 1$, then the space of left-invariant orders on~$G$ has no isolated points. We prove the same for the space of bi-invariant orders:

\begin{prop} \label{BONoIsolated}
If $G$ is a locally nilpotent group, and $G$ is not an abelian group of rank\/~$\le 1$, then the space of bi-invariant orders on~$G$ has no isolated points.
\end{prop}

We also prove some variants of the following recent result.

\begin{thm}[{T.\,Koberda \cite{Koberda-Faithful}}]
If $G$ is a finitely generated group that is residually torsion-free nilpotent, then the natural action of\/ $\Aut(G)$ on\/ $\LO(G)$ is faithful.
\end{thm}

Our modifications of Koberda's theorem replace the automorphism group of~$G$ with the larger group of abstract commensurators. Before stating these results, we present some background material.

\begin{defn}
Let $G$ be a group.
	\begin{itemize} \itemsep=\smallskipamount
	\item A \emph{commensuration} of~$G$ is an isomorphism $\phi \colon G_1 \to G_2$, where $G_1$ and~$G_2$ are finite-index subgroups of~$G$.
	\item Two commensurations $\phi \colon G_1 \to G_2$ and $\phi' \colon G_1' \to G_2'$ are \emph{equivalent} if there exists a finite-index subgroup~$H$ of $G_1 \cap G_1'$, such that $\phi$ and~$\phi'$ have the same restriction to~$H$.
	\item The equivalence classes of the commensurations of~$G$ form a group that is denoted $\Comm(G)$. It is called the \emph{abstract commensurator} of~$G$.
	\end{itemize}
\end{defn}

In general, there is no natural action of $\Comm(G)$ on $\LO(G)$, but the following observation provides a special case in which we do have such an action:

\begin{lem} \label{CommGActLO}
Let $G$ be a torsion-free, locally nilpotent group.
	\begin{enumerate}
	\item \label{CommGActLO-bij}
	If $H$ is any finite-index subgroup of~$G$, then the natural restriction map\/ $\LO(G) \to \LO(H)$ is a bijection.
	\item \label{CommGActLO-act}
	Therefore, there is a natural action of\/ $\Comm(G)$ on $\LO(G)$.
	\end{enumerate}
\end{lem}

This action allows us to state the following variant of Koberda's theorem that replaces $\Aut(G)$ with $\Comm(G)$, but, unfortunately, requires $G$ to be locally nilpotent, not just residually nilpotent. 

\begin{prop} \label{NilpCommFaithful}
If $G$ is a nonabelian, torsion-free, locally nilpotent group, then the action of\/ $\Comm(G)$ on\/ $\LO(G)$ is faithful.
\end{prop}

\begin{rem}
The \namecref{NilpCommFaithful} assumes that $G$ is nonabelian. When $G$ is abelian, \cref{WhenAbelFaithful} shows that the action of $\Comm(G)$ is faithful iff the subgroup~$G^n$ of $n^{\text{th}}$ powers has infinite index in~$G$, for all $n \ge 2$.
\end{rem}

For non-nilpotent groups, $\Comm(G)$ may not act on $\LO(G)$, but it does act on a certain space $\VLO(G)$ that contains $\LO(G)$ \csee{VirtLOSect}. (It is a space of left-invariant orders on finite-index subgroups of~$G$.) This action allows us to state the following generalization of Koberda's Theorem:

\begin{cor} \label{CommMovesLOG}
If $G$ is a nonabelian group that is residually locally torsion-free nilpotent, and $\alpha$ is any nonidentity element of\/ $\Comm(G)$, then there exists ${\prec} \in \LO(G)$, such that ${\prec}^\alpha \neq {\prec}$. 
\end{cor}

The following is an immediate consequence:

\begin{cor} \label{CommFaithful}
If $G$ is a nonabelian group that is residually locally torsion-free nilpotent, then the action of\/ $\Comm(G)$ on\/ $\VLO(G)$ is faithful.
\end{cor}

There is a natural action of the outer automorphism group $\Out(G)$ on $\BO(G)$, because every inner automorphism acts trivially on this space.
T.\,Koberda \cite[\S6]{Koberda-Faithful} observed that if $G$ is the fundamental group of the Klein bottle, then this action is not faithful. (Note that this group~$G$ is solvable. In fact, it is polycyclic and metabelian). In \cref{NonFaithfulSect}, we improve this example by exhibiting finitely generated, nilpotent groups for which the action is not faithful. (Like Koberda's, our groups are polycyclic and metabelian.)

\begin{ack}
I would like to thank the participants in the workshop on ``Ordered Groups and Topology'' (Banff International Research Station, Alberta, Canada, February 12--17, 2012) for the stimulating lectures and discussions that instigated this research.
I would also like to thank the mathematics department of the University of Chicago for its hospitality during the preparation of this manuscript.
An anonymous referee also deserves thanks for providing an extraordinarily prompt report that included helpful comments on the exposition.
\end{ack}

\section{Preliminaries} \label{PrelimSect}

\subsection{Preliminaries on nilpotent groups}

\begin{defn}[{}{\cite[p.~85 (1)]{LennoxRobinson-ThyInfSolvGrps}}]
If $G$ is a solvable group, then, by definition, there is a subnormal series  
	\begin{align*}
	 G = G_r \lamron G_{r-1} \lamron \cdots \lamron G_1 \lamron G_0 = \{e\} 
	,  \end{align*}
such that each quotient $G_i/G_{i-1}$ is abelian. The \emph{Hirsch rank} of~$G$ is sum of the (torsion-free) ranks of these abelian groups.  (This is also known as the \emph{torsion-free rank} of~$G$.) More precisely,
	$$ \rank G = \sum_{i = 1}^r \dim_{\rational} \bigl( (G_i/G_{i-1}) \otimes \rational \bigr) .$$
It is not difficult to see that this is independent of the choice of the subnormal series.
\end{defn}

\begin{notation}
Let $S$ be a subset of a group~$G$.
\noprelistbreak
	\begin{itemize}
	\item As usual, we use $\langle S \rangle$ to denote the smallest subgroup of~$G$ that contains~$S$.
	\item We let
		$$ \isol{S}_G = \bigset{x \in G }{ \exists m \in \integer^+, x^m \in \langle S \rangle }.$$
	When the group~$G$ is clear from the context, we usually omit the subscript, and write merely $\isol{S}$.
	\end{itemize}
\end{notation}

\begin{lem}[{}{\cite[2.3.1(i)]{LennoxRobinson-ThyInfSolvGrps}}]
If $G$ is a locally nilpotent group, then $\isol{S}$ is a subgroup of~$G$, for all $S \subseteq G$.
\end{lem}

\begin{rem}
A subgroup~$H$ is said to be \emph{isolated} if $H = \isol{H}$, but we do not need this terminology.
\end{rem}

We provide a proof of the following well-known fact, because we do not have a convenient reference for it.

\begin{lem} \label{Rank(Abelianization)}
If $G$ is a finitely generated, nilpotent group, and $\rank G \ge 2$, then 
	$$\rank \bigl( G/ [G,G] \bigr) \ge 2 .$$
\end{lem}

\begin{proof}
For every proper subgroup~$H$ of~$G$, such that $\isol{H} = H$, we have $N_G(H) = \isol{N_G(H)}$ \cite[2.3.7]{LennoxRobinson-ThyInfSolvGrps} and $N_G(H) \supsetneq H$ \cite[Cor.~10.3.1, p.~154]{Hall-ThyGrps}. This implies $\rank N_G(H) > \rank H$, so, for any $g \in G$, there is a subnormal series
	$$ \isol{g} = G_0 \normal G_1 \normal \cdots \normal G_{s-1} \normal G_s = G ,$$
with $\isol{G_i} = G_i$ for every~$i$. By refining the series, we may assume $1 + \rank G_{s-1} = \rank G$.
Then $G/ G_{s-1}$ is a torsion-free, nilpotent group of rank~1, and is therefore abelian (cf.\ \cite[2.3.9(i)]{LennoxRobinson-ThyInfSolvGrps}), so $[G,G] \subseteq G_{s-1}$. Since $G_{s-1}$ also contains~$g$, we conclude that $\isol{g, [G,G]} \neq G$. This implies $\rank \bigl( G/ [G,G] \bigr) \ge 2$, because $g$ is an arbitrary element of~$G$.
\end{proof}

\subsection{Preliminaries on ordered groups}

\begin{defn}[{}{\cite[pp.~29, 31, and 34]{KopytovMedvedev-ROGrps}}]
Let $\prec$ be a left-invariant order on a group~$G$.
	\begin{itemize}
	\item A subgroup $C$ of~$G$ is \emph{convex} if, for all $c,c' \in C$, and all $g \in G$, such that $c \prec g \prec c'$, we have $g \in C$.
	\item We say that $C_2/C_1$ is a \emph{convex jump} if $C_1$ and~$C_2$ are convex subgroups, and $C_1$ is the maximal convex proper subgroup of~$C_2$.
	\item A convex jump $C_2/C_1$ is \emph{Archimedean} if there is a nontrivial homomorphism $\varphi \colon C_2 \to \real$, such that, for all $c,c' \in C_2$, we have $\varphi(c) < \varphi(c') \Rightarrow c \prec c'$. (Since $C_1$ is the maximal convex subgroup of~$C_2$, it is easy to see that this implies $\ker \varphi = C_1$.)
	\end{itemize}
\end{defn}

\begin{rem}[{}{\cite[Thm.~2.1.1, p.~31]{KopytovMedvedev-ROGrps}}] \label{SetOfConvexSubgrps}
 If $\prec$ is a left-invariant order on a group~$G$, then it is easy to see that the set of convex subgroups is totally ordered under inclusion, and is closed under arbitrary intersections and unions. Therefore, each element~$g$ of~$G$ determines a convex jump $C_2(g)/C_1(g)$, defined by letting
 	\begin{itemize}
	\item $C_2(g)$ be the (unique) smallest convex subgroup of~$G$ that contains~$g$, 
	and
	\item $C_1(g)$ be the (unique) largest convex subgroup of~$G$ that does not contain~$g$.
	\end{itemize}
 \end{rem}

The following easy observation is well known:

\begin{lem}[{cf.\ \cite[Lem.~5.2.1, p.~132]{KopytovMedvedev-ROGrps}}] \label{ChangeOnQuot}
Let 
	\begin{itemize}
	\item $\prec$ be a left-invariant order on a group~$G$,
	\item $C_1$ and~$C_2$ be convex subgroups of~$G$, such that $C_1 \normal C_2$,
	\item $\overline{\phantom{x}} \colon C_2 \to C_2/C_1$ be the natural homomorphism,
	and
	\item $\ll$ be a left-invariant order on the group $C_2/C_1$.
	\end{itemize}
Then there is a\/ {\upshape(}unique\/{\upshape)} left-invariant order~$\precstar$ on~$G$, such that, for $g \in G$, we have
	$$ g \succstar e \iff
	\begin{cases}
	\overline{g} \gg \overline{e} & \text{if $g \in C_2$ and $g \notin C_1$}, \\
	g \succ e & \text{otherwise}
	. \end{cases}
	$$
\end{lem}

\begin{defn}
The construction in \cref{ChangeOnQuot} is called \emph{changing $\prec$ on $C_2/C_1$}.
\end{defn}

\begin{lem}[{}{\cite[Thm.~2.4.2, p.~41]{KopytovMedvedev-ROGrps}}] \label{Nilp->Conradian}
If $\prec$ is a left-invariant order on a group~$G$ that is locally nilpotent, then every convex jump is Archimedean.
\end{lem}

\begin{rem}
A left-invariant order is said to be \emph{Conradian} if all of its convex jumps are Archimedean, but we do not need this terminology.
\end{rem} 

\begin{lem}[{}{\cite[p.~227]{BludovGlassRhemtulla-JumpsCentral}}] \label{CentralJumps}
If $\prec$ is a bi-invariant order on a group~$G$ that is locally nilpotent, then every convex jump $C_2/C_1$ is central. {\upshape(}This means $[G, C_2] \subseteq C_1$.{\upshape)} Therefore, every convex subgroup of~$G$ is normal.
\end{lem}

\begin{lem}[{}{cf.\ \cite[Prop.~1.7]{Sikora-TopOrders}}] \label{AbelIsol->Rank1}
Let $G$ be a nontrivial, abelian group. If the space of bi-invariant orders on~$G$ has an isolated point, then $\rank G = 1$.
\end{lem}

\begin{thm}[{Rhemtulla \cite[Cor.~3.6.2, p.~66]{KopytovMedvedev-ROGrps}}] \label{ExtendLONilp}
If $G$ is a torsion-free, locally nilpotent group, then any left-invariant order on any subgroup of~$G$ extends to a left-invariant order on all of~$G$.
\end{thm}

\section{Topology of the space of bi-invariant orders} \label{TopBiInvtSect}

In this section, we prove \textbf{\cref{BONoIsolated}}:
\noprelistbreak
	\begin{quote} 
	\it If $G$ is a locally nilpotent group, and $G$ is not an abelian group of rank\/~$\le 1$, then the space of bi-invariant orders on~$G$ has no isolated points. 
	\end{quote}

\begin{proof}[\bf Proof of \cref{BONoIsolated}]
Suppose $\prec$ is an isolated point in the space of bi-invariant orders on~$G$. By definition of the topology on $\BO(G)$, this means there is a finite subset~$S$ of~$G$, for which $\prec$ is the unique bi-invariant order on~$G$ that satisfies $g \succ e$ for all $g \in S$.

If we change $\prec$ on any convex jump $C_i/C_{i-1}$, then the resulting left-invariant order will actually be bi-invariant (since \cref{CentralJumps} tells us that the jump is central). Therefore, the fact that $\prec$ is isolated implies that it has only finitely many convex jumps. (Indeed, every convex jump must be determined by some element of the finite set~$S$.) Thus, we may let
	\begin{align} \label{ConvexChain} 
	 G = C_r \supsetneq C_{r-1} \supsetneq \cdots \supsetneq C_1 \supsetneq C_0 = \{e\} 
	 \end{align}
be the chain of convex subgroups. From \cref{CentralJumps}, we know that this is a central series (so $G$ is nilpotent, not just locally nilpotent, as originally assumed). The fact that $\prec$ is isolated also implies that each convex jump $C_i/C_{i-1}$ has an isolated left-invariant order. Then, since the jump is a nontrivial abelian group, \cref{AbelIsol->Rank1} tells us that $\rank (C_i/C_{i-1}) = 1$.
So $\rank G = r$.

Let
	\begin{align} \label{UCC}
	 G = Z_c \lamron Z_{c-1} \lamron \cdots \lamron Z_1 \lamron Z_0 = \{e\} 
	 \end{align}
be the the upper central series of~$G$. (It is defined by letting $Z_i/Z_{i-1}$ be the center of $G/Z_{i-1}$.) Then $c$ is the nilpotence class of~$G$. It is well known that 
$c < \rank G$ (for example, this follows from \cref{Rank(Abelianization)}),
which means $c <  r$. So there is some~$k$ with $C_k \neq Z_k$, and we may assume $k$~is minimal. Since $\{C_i\}$ is a central series, and $\{Z_i\}$ is the upper central series, we have 
$C_k \subseteq Z_k$ \cite[Thm.~10.2.2]{Hall-ThyGrps}. 
Therefore, there exists some $z \in Z_k$, such that $z \notin C_k$.

Choose $\ell$ minimal with $z \in C_\ell$. (Note that $k \le \ell-1$.) Since \pref{ConvexChain} and \pref{UCC} are central series, we have $[G, C_{\ell-1}] \subseteq C_{\ell-2}$ and 
	$$ [G, z] \subseteq [G,Z_k] \subseteq Z_{k-1} = C_{k-1} \subseteq C_{\ell-2} .$$
Therefore $[ G, \langle C_{\ell-1}, z \rangle ] \subseteq C_{\ell-2}$.
Since $\rank ( C_\ell/C_{\ell-1} ) = 1$, we know that $C_\ell / \langle C_{\ell-1}, z \rangle$ is a torsion group, so this implies that $[G, C_\ell] \subseteq C_{\ell-2}$ \cite[2.3.9(vi)]{LennoxRobinson-ThyInfSolvGrps}. This means that $G$ centralizes $C_\ell/ C_{\ell-2}$, so changing $\prec$ on $C_\ell/ C_{\ell-2}$ will result in another bi-invariant order. Since $C_\ell/ C_{\ell-2}$ is an abelian group of rank~$2$, \cref{AbelIsol->Rank1} tells us that it has no isolated order, so we conclude that $\prec$ is not isolated. This is a contradiction.
\end{proof}

\begin{cor} \label{BOCantor}
Let $G$ be a locally nilpotent group that is not an abelian group of rank\/~$\le 1$. If $G$ is countable and torsion-free, then the space of bi-invariant orders on~$G$ is homeomorphic to the Cantor set.
\end{cor}

\section{\texorpdfstring{The action of $\Comm(G)$ on $\LO(G)$ when $G$ is nilpotent}
{The action of Comm(G) on LO(G) when G is nilpotent}} \label{CommGNilp}

Combining the two parts of the following observation yields \textbf{\cref{CommGActLO}}\pref{CommGActLO-bij}.

\begin{obs} \label{LOG->LOH}
Assume $H$ is a subgroup of a torsion-free group~$G$, and let $\eta \colon \LO(G) \to \LO(H)$ be the natural restriction map.
\noprelistbreak
	\begin{enumerate}
	\item \label{LOG->LOH-inj}
	If $H$ has finite index in~$G$, then $\eta$ is injective. (To see this, let ${\prec} \in \LO(G)$ and note that if $x \in G$, then there is some $n \in \integer^+$, such that $x^n \in H$. We have $x \succ e \iff x^n \succ e$, so the positive cone of~$\prec$ is determined by its restriction to~$H$. Combine this with the fact that any left-invariant order is determined by its positive cone.) 
	\item \label{LOG->LOH-surj}
	If $G$ is locally nilpotent, then $\eta$ is surjective (see \cref{ExtendLONilp}).
	\end{enumerate}
\end{obs}

In the remainder of this section, we prove \cref{CommGKernel}, which contains \textbf{\cref{NilpCommFaithful}}:
\noprelistbreak
	\begin{quote}
	\it If $G$ is a nonabelian, torsion-free, locally nilpotent group, then the action of\/ $\Comm(G)$ on\/ $\LO(G)$ is faithful.
	\end{quote}

\begin{notation}
Assume $G$ is a torsion-free, abelian group. 
\noprelistbreak
	\begin{itemize}
	\item For $n \in \integer$, we let $G^n = \{\, g^n \mid g \in G \,\}$. This is a subgroup of~$G$ (since $G$ is abelian).
	\item For $p/q \in \rational$, such that $G^p$ and~$G^q$ have finite index in~$G$, we define $\tau^{p/q} \in \Comm(G)$ by $\tau^{p/q}(g^q) = g^p$ for $g \in G$.
	\end{itemize}
\end{notation}

\begin{prop} \label{CommGKernel}
Let $G$ be a torsion-free, locally nilpotent group.
	\begin{enumerate}
	\item If $G$ is not abelian, then the action of\/ $\Comm(G)$ on\/ $\LO(G)$ is faithful.
	\item If $G$ is abelian, then the kernel of the action is
	$$ \bigset{ \tau^{p/q} }{ \begin{matrix} 
	\text{$p,q \in \integer^+$, such that} \\ 
	\text{$G^p$ and~$G^q$ have finite index in~$G$}
	\end{matrix} } 
	.$$
	\end{enumerate}
\end{prop}

\begin{proof}[Proof \rm (cf.\ proof of {\cite[Thm.~4.1]{Koberda-Faithful}})]
For $p,q \in \integer^+$, $g \in G$, and ${\prec} \in \LO(G)$, we have $g^p \succ e \iff g^q \succ e$. Therefore,  $\tau^{p/q}$ acts trivially on $\LO(G)$ if it exists.
To complete the proof, we wish to show that the kernel is trivial if $G$ is not abelian, and that every element of  the kernel is of the form $\tau^{p/q}$ if $G$ is abelian.

Let $\alpha$ be an element of the kernel. We consider three cases.

\setcounter{case}{0}

\begin{case}
Assume there exists $g \in \dom \alpha$, such that $g^\alpha \notin \isol{g}$.
\end{case}
Let $H = \langle g, g^\alpha \rangle$ and $\overline{H} = H/\isol{[H,H]}_H$. There is a left-invariant order~$\ll$ on the abelian group~$\overline{H}$, such that $\isol{\overline{g}}$ is a convex subgroup. Then $\isol{\overline{g}}/ \isol{\overline{e}}$ is a convex jump. Since $\overline{g^\alpha} \notin \isol{\overline{g}}$ \csee{Rank(Abelianization)}, this implies that  $g$ and~$g^\alpha$ determine different convex jumps of~$\ll$. By applying \cref{ExtendLONilp}, we see that there is a left-invariant order~$\prec$ on~$G$, such that $g$ and~$g^{\alpha}$ determine two different convex jumps. Reversing the order on the convex jump containing~$g^\alpha$ yields a second left-invariant order, and it is impossible for $\alpha$ to fix both of these orders. This contradicts the fact that $\alpha$ is in the kernel of the action.

\begin{case} \label{NilpCommFaithfulPf-Abel}
Assume $G$ is abelian, and $g^\alpha \in \isol{g}$, for all $g \in \dom\alpha$.
\end{case}
The assumption means that every element of $\dom \alpha$ is an eigenvector for the action of $\alpha$ on the vector space $G \otimes \rational$. Since $\dom \alpha$ is a subgroup, we know that it is closed under addition as a subset of $G \otimes \rational$. This implies that all of $\dom \alpha$ is in a single eigenspace. Then, since $\dom \alpha$ has finite index, we conclude that $G \otimes \rational$ is a single eigenspace, so there is some $p/q \in \rational$, such that we have $\alpha(v) = (p/q)v$ for all $v \in G \otimes \rational$. In other words, $\alpha(g) = \tau^{p/q}(g)$ for all $g \in G$.

We must have $p/q \in \rational^+$, since $\tau^{p/q} = \alpha$ acts trivially on $\LO(G)$, and $g \succ e \implies g^{-1} \not\succ e$. Also, since $\tau^{p/q} = \alpha \in \Comm(G)$, we know that the domain~$G^q$ and range~$G^p$ of $\tau^{p/q}$ have finite index in~$G$ (assuming, without loss of generality, that $p/q$ is in lowest terms).

\begin{case}
Assume $G$ is nonabelian, and $g^\alpha \in \isol{g}$, for all $g \in \dom \alpha$.
\end{case}
For each nontrivial $g \in \dom \alpha$, the assumption tells us there exists $r(g) \in \rational$, such that $\alpha(g) = g^{r(g)}$. 
The eigenvector argument of the preceding case 
shows that $r(g) = r(h)$ whenever $g$ commutes with~$h$. 
However, for all $g, h \in G$, there is some nontrivial $z \in G$ that commutes with both $g$ and~$h$ (since $G$ is locally nilpotent), so $r(g) = r(z) = r(h)$. Therefore $r(g) = r$ is independent of~$g$.

On the other hand, since $G$ is locally nilpotent, but not abelian, we may choose $g,h \in \dom \alpha$, such that $\langle g,h \rangle$ is nilpotent of class~$2$. This means that $[g,h]$ is a nontrivial element of the center of $\langle g,h \rangle$. Then
	$$ [g,h]^r = [g,h]^\alpha = [g^\alpha, h^\alpha] = [g^r, h^r] = [g,h]^{r^2} ,$$
so $r = r^2$. Hence $r = 1$, so $\alpha(g) = g^r = g^1 = g$ for all $g \in G$.
\end{proof}

\begin{cor} \label{WhenAbelFaithful}
Assume $G$ is a torsion-free, locally nilpotent group. Then the action of\/ $\Comm(G)$ on\/ $\LO(G)$ is faithful iff either
	\begin{enumerate}
	\item $G$ is not abelian,
	or
	\item $G^n$ has infinite index in~$G$ for all $n \ge 2$,
	or
	\item $G = \{e\}$ is trivial.
	\end{enumerate}
\end{cor}

\begin{rem}
The proof of \cite[Thm.~4.1]{Koberda-Faithful} assumes that $G$ is finitely generated, but this was omitted from the statement of the result. (The group~$\rational$ has infinitely many automorphisms~$\tau^{p/q}$, but only two left-invariant orders, so it provides a counterexample to the theorem as stated.)
\end{rem}

\section{The space of virtual left-orders} \label{VirtLOSect}

\begin{defn}
Note that if $H_1$ is a subgroup of a group~$H_2$, then we have a natural restriction map $\LO(H_2) \to \LO(H_1)$. Therefore, we may define the direct limit
	$$ \VLO(G) = \lim_{\longrightarrow} \LO(H) ,$$
where the limit is over all finite-index subgroups~$H$ of~$G$. An element of $\VLO(G)$ can be called a \emph{virtual left-invariant order} on~$G$.
\end{defn}

\begin{rems} \ 
\noprelistbreak
	\begin{enumerate}
	\item There is a natural action of $\Comm(G)$ on $\VLO(G)$.
	\item \fullcref{LOG->LOH}{inj} tells us that the inclusion $\LO(G) \hookrightarrow \VLO(G)$ is injective, so we can think of $\LO(G)$ as a subset of $\VLO(G)$.
	\end{enumerate}
\end{rems}

We now have the notation to prove \textbf{\cref{CommMovesLOG}}:
	\begin{quote} \it
	If $G$ is a nonabelian group that is residually locally torsion-free nilpotent, and $\alpha$ is a nonidentity element of\/ $\Comm(G)$, then ${\prec}^\alpha \neq {\prec}$, for some ${\prec} \in \LO(G)$.
	\end{quote}

\begin{proof}
Suppose $\alpha$ fixes every element of $\LO(G)$. Since reversing a left-invariant order on any convex jump yields another left-invariant order, it is clear that $\alpha$ must fix every convex jump of every left-invariant order. More precisely, 
	\begin{align} \label{alphaFixRelConvex}
	\begin{matrix}
	\text{if $C$ is any convex subgroup of~$G$ (with respect to some}
	\\ \text{left-invariant order), then $\alpha^{-1}(C) = C \cap \dom \alpha$.}
	\end{matrix}
	\end{align}

Choose $g \in \dom \alpha$, such that $g^\alpha \neq g$. Then we may choose a nonabelian, torsion-free, locally nilpotent quotient $G/N$ of~$G$, such that $g^\alpha \notin g N$. Since $G/N$ is torsion-free and locally nilpotent, we know that it has a left-invariant order. We can extend this to a left-invariant order on~$G$  (by choosing any left-invariant order on the subgroup~$N$). Then $N$ is a convex subgroup for this order, so, from \pref{alphaFixRelConvex}, we know that $\alpha$ induces a well-defined $\overline{\alpha} \in \Comm(G/N)$. Then \cref{NilpCommFaithful} tells us there exists ${\ll} \in \LO(G/N)$, such that ${\ll}^\alpha \neq {\ll}$. Extend $\ll$ to a left-invariant order~$\prec$ on~$G$ (by choosing any left-invariant order on the subgroup~$N$). Then ${\prec}^\alpha \neq \prec$.
\end{proof}

\section{Non-faithful actions on the space of bi-invariant orders} \label{NonFaithfulSect}

In this section, we provide examples of torsion-free, nilpotent groups for which there is a nontrivial commensuration that acts trivially on the space of bi-invariant orders.

\begin{eg}
For $r \in \integer^+$, let 
	$$G_r = \langle\, x,y,z \mid [x,y] = z^r, [x,z] = [y,z] = e \,\rangle .$$ 
(Then $G_1$ is the discrete Heisenberg group, and $G_r$ is a finite-index subgroup of it.)
Since $\langle z \rangle = Z(G)$, it is easy to see that $z \in \isol{N}$, for every nontrivial, normal subgroup~$N$ of~$G$. Hence, if we define an automorphism $\alpha \colon G_r \to G_r$ by 
	$$ \text{$\alpha(x) = xz$, \ $\alpha(y) = y$, \ and \ $\alpha(z) = z$,} $$
then $\alpha$ acts trivially on $\BO(G_r)$. However, $\alpha$ is outer if $r > 1$ (since $z \notin \langle z^r \rangle = [G_r,G_r]$). Thus, $\Out(G_r)$ does not act faithfully on $\BO(G_r)$ when $r > 1$. 

On the other hand, it is easy to see that $\Out(G_1)$ does act faithfully on $\BO(G_1)$. This means that deciding whether $\Out(G)$ acts faithfully is a rather delicate question --- the answer can be different for two groups that are commensurable to each other.
\end{eg}

Here is an example where we get the same answer for all torsion-free, nilpotent groups that are commensurable:

\begin{eg} \label{OutNotFaithfulGoodEg}
Let $G = \integer \ltimes \integer^3$, where $\integer$ acts on $\integer^3$ via the matrix
	{\smaller[4]$\begin{bmatrix} 1 & 0 & 0 \\ 1 & 1 & 0 \\ 0 & 1 & 1 \end{bmatrix}$}.
In other words, 
	$$ G = 
	\bigl\langle x, y,z,w \mid [x,w] = y, \ [y,w] = z, \ \text{other commutators trivial} \bigr\rangle .$$
Choose $r \in \integer \smallsetminus \{0\}$, and define $\alpha \in \Aut(G)$ by 
	$$ \text{$x^\alpha =  x z^r$, \ $y^\alpha = y$, \ $z^\alpha = z$, \ $w^\alpha = w$.} $$
We claim $\alpha$ is an outer automorphism of~$G$ that acts trivially on $\BO(G)$.
\end{eg}

\begin{proof}
Note that $[x, hw^n] = [x,w^n] \in y^n \langle z \rangle$ for all $h \in \langle x,y,z\rangle$ and all $n \in \integer$. This implies that if $g \in G$, and $[x,g] \neq e$, then $[x,g] \notin \langle z \rangle$. Since $x^\alpha \in x \, \langle z \rangle$, we conclude that $\alpha$ is outer.

Let ${\prec} \in \BO(G)$, and let $C$ be the minimal nontrivial convex subgroup of~$G$. From \cref{CentralJumps}, we know $C$ is a subgroup of $Z(G)$. Since $Z(G) = \langle z \rangle$ has rank one, we conclude that $C = \langle z \rangle$ is the (unique) minimal nontrivial convex subgroup. Since $\alpha$ centralizes both $\langle z \rangle$ and $G/ \langle z \rangle$, this implies that $\alpha$ centralizes every convex jump $C_2/C_1$. Therefore ${\prec}^\alpha = \alpha$. Since $\prec$ is an arbitrary bi-invariant order, we conclude that $\alpha$ acts trivially on $\BO(G)$.
\end{proof}

\section{Nilpotent Lie groups and left-invariant orders} \label{LieGrpLO}

It is easy to see that if $\prec$ is a left-order on the abelian group $G = \integer^n$, then there is a nontrivial linear function $\varphi \colon \real^n \to \real$, such that, for all $x,y \in \integer^n$, we have
	$$ \varphi(x) < \varphi(y) \implies x \prec y .$$
We will generalize this observation in a natural way to any finitely generated, nilpotent group~$G$, by choosing an appropriate embedding of~$G$ in a connected Lie group \csee{GisLatt,order->homo}.

\subsection{Preliminaries on discrete subgroups of nilpotent Lie groups}

\begin{defn}
A topological space is \emph{1-connected} if it is connected and simply connected.
\end{defn}

\begin{prop}[{}{\cite[Thm.~2.18, p.~40, and Cor.~2, p.~34]{RaghunathanBook}}] \label{GisLatt}
Every finitely generated, torsion-free, nilpotent group is isomorphic to a discrete, cocompact subgroup of a 1-connected, nilpotent Lie group~$\GG$. Furthermore, $\GG$ is unique up to isomorphism.
\end{prop}

\begin{rem}[{}{\cite[Thm.~2.10, p.~32]{RaghunathanBook}}] \label{discrete->fg}
Conversely, every discrete subgroup of a 1-connected, nilpotent Lie group\/~$\GG$ is finitely generated.
\end{rem}

\begin{prop}[{}{\cite[Thm.~2.11, p.~33]{RaghunathanBook}}] \label{MalcevSuperrig}
Suppose
	\begin{itemize}
	\item $\GG_1$ and~$\GG_2$ are $1$-connected, nilpotent Lie groups,
	\item $G$ is a discrete, cocompact subgroup of\/~$\GG_1$,
	and
	\item $\rho \colon G \to \GG_2$ is a homomorphism.
	\end{itemize}
Then $\rho$ extends\/ {\upshape(}uniquely\/{\upshape)} to a continuous homomorphism $\widehat\rho \colon \GG_1 \to \GG_2$.
\end{prop}

\begin{defn}[{}{\cite[Defn.~3.1]{Witte-SyndHulls}}]
Let $G$ be a discrete subgroup of a Lie group~$\GG$. A closed, connected subgroup~$\HH$ of~$\GG$ is a \emph{syndetic hull} of~$G$ if 
	 $G \subseteq \HH$
	and
	 $\HH/G$ is compact.
\end{defn}

\begin{prop}[{}cf.\ {\cite[Prop.~2.5, p.~31]{RaghunathanBook}}]
If\/ $\GG$ is a 1-connected, nilpotent Lie group, then every discrete subgroup of\/~$\GG$ has a unique syndetic hull.
\end{prop}

\subsection{Description of left-invariant orders on nilpotent groups}

\begin{prop} \label{order->homo}
Assume
	\begin{itemize}
	\item $\GG$ is a 1-connected, nilpotent Lie group,
	\item $G$ is a nontrivial, discrete, cocompact subgroup of\/~$\GG$,
	and
	\item $\prec$ is a left-invariant order on\/~$G$.
	\end{itemize}
Then there is a nontrivial, continuous homomorphism $\varphi \colon \GG \to \real$, such that, for all $x,y \in G$, we have
	$$ \varphi(x) < \varphi(y) \implies x \prec y .$$
Furthermore, $\varphi$ is unique up to multiplication by a positive scalar.
\end{prop}

\begin{proof}
Since $G$ is finitely generated \csee{discrete->fg}, it is easy to see that $G$ has a maximal convex subgroup~$C$ (cf.\ \cref{SetOfConvexSubgrps}), so $G/C$ is a convex jump. Also, since $G$ is nilpotent, we know that every convex jump is Archimedean \csee{Nilp->Conradian}. Therefore, there is a nontrivial homomorphism $\varphi_0 \colon G \to \real$, such that $\varphi_0(x) < \varphi_0(y) \Rightarrow x \prec y$. (Furthermore, this homomorphism is unique up to multiplication by a positive scalar \cite[Prop.~2.2.1, p.~34]{KopytovMedvedev-ROGrps}.) From \cref{MalcevSuperrig}, we know that $\varphi_0$ extends (uniquely) to a continuous homomorphism $\varphi \colon \GG \to \real$.
\end{proof}

The results in previous sections were originally obtained by using the following structural description of each left-invariant order on any finitely generated, nilpotent group.

\begin{cor} \label{LieDescripLONilp}
Assume 
	\begin{itemize}
	\item $\GG$ is a 1-connected, nilpotent Lie group,
	\item $G$ is a discrete, cocompact subgroup of\/~$\GG$,
	and 
	\item $\prec$ is a left-invariant order on~$G$.
	\end{itemize}
Then there exist:
	\begin{itemize}
	\item a subnormal series\/ $\GG = \CC_r \lamron \CC_{r-1} \lamron \cdots \lamron \CC_1  \lamron \CC_0 = \{e\}$ of closed, connected subgroups of\/~$\GG$,
	and
	\item for each $i \in \{1,\ldots,r\}$, a nontrivial, continuous homomorphism $\varphi_i \colon \CC_i \to \real$,
	\end{itemize}
such that, for $1 \le i \le r$:
	\begin{enumerate}
	\item for all $x,y \in G \cap \CC_i$, we have $\varphi_i(x) < \varphi_i(y) \implies x \prec y$,
	\item \label{LieDescripLONilp-cocpct}
	$G \cap \CC_i$ is a cocompact subgroup of\/~$\CC_i$,
	and
	\item \label{LieDescripLONilp-intersect}
	$G \cap \ker \varphi_i = G \cap \CC_{i-1}$.
	\end{enumerate}
Furthermore, the subgroups\/ $\CC_1,\ldots,\CC_r$ are unique, and each homomorphism $\varphi_i$ is unique up to multiplication by a positive scalar.
\end{cor}

\begin{proof}
Let $\varphi \colon \GG \to \real$ be the homomorphism provided by \cref{order->homo}, and let $\CC$ be the syndetic hull of $G \cap \ker \varphi$. By induction on $\dim \GG$, we can apply the \namecref{LieDescripLONilp} to~$\CC$, obtaining
	\begin{itemize}
	\item a chain $\CC = \CC_{r-1} \lamron \CC_{r-2} \lamron \cdots \lamron \CC_1  \lamron \CC_0 = \{e\}$ of closed, connected subgroups of~$\CC$,
	and
	\item for each $i \in \{1,\ldots,r-1\}$, a nontrivial, continuous homomorphism $\varphi_i \colon \CC_i \to \real$.
	\end{itemize}
To complete the construction, let $\CC_r = \GG$ and $\varphi_r = \varphi$.
\end{proof}

\begin{rems} \label{LieDescripRems} \ 
\noprelistbreak
	\begin{enumerate}
	\item It is not difficult to show that each quotient $\CC_i/\CC_{i-1}$ is abelian.
	\item  \label{LieDescripRems-normal} 
	In the setting of \cref{LieDescripLONilp}, the order~$\prec$ is bi-invariant iff $\CC_i$ and~$\ker \varphi_i$ are normal subgroups of~$\GG$, for $1 \le i \le r$.
	\item The converse of \cref{LieDescripLONilp} is true: if subgroups $\CC_i$ and homomorphisms~$\varphi_i$ are provided that satisfy \pref{LieDescripLONilp-cocpct} and~\pref{LieDescripLONilp-intersect}, then the positive cone of a left-invariant order~$\prec$ can be defined by prescribing:
	$$ x \succ e \iff \varphi_i(x) > 0 ,$$
where $i$~is chosen so that $x \in \CC_i \smallsetminus \CC_{i-1}$.
	\end{enumerate}
\end{rems}

\end{document}